\documentclass[12pt]{article}

\usepackage{amsmath, amsfonts, amssymb}
\usepackage{mathrsfs}
\usepackage{theorem}
\usepackage{bm}
\usepackage{nicefrac}

\allowdisplaybreaks

\newtheorem{mythm}{Theorem}[section]
\newtheorem{myprop}[mythm]{Proposition}
\newtheorem{mylem}[mythm]{Lemma}

{\newtheorem{myrem}[mythm]{Remark}}
{}

\newcommand{\ra}{\rightarrow}
\newcommand{\dis}{\displaystyle}
\def\R{\mathbb R}

\def\S{\mathcal S}

\def\F{\mathscr F}
\def\d{\text{\rm{d}}}
\def\E{\mathbb E}

\def\e{\text{\rm{e}}}

\def\La{\Lambda}
\def\veps{\varepsilon}

\def\wt{\widetilde}

\newcommand{\fin}{\hfill $\square$\par}
\newenvironment{proof}{{\noindent\it Proof.}\ }{\hfill $\square$\par}

\numberwithin{equation}{section}

\def\D{\mathcal{D}}
\def\J{\mathcal{J}}
\def\Di{\D^{(i)}}
\def\Ji{\J^{(i)}}
\def\A{\mathcal{A}}
\def\ii{\hat{\imath}}
\def\sgn{\mathrm{sgn}}

\begin{document}

\title{Long time behavior of L\'{e}vy-driven Ornstein-Uhlenbeck process with regime-switching \footnote{Supported in
 part by NNSFs of China (Nos. 11771327, 11701588, 11431014, 11831014)}}
\author{ Zhong-Wei Liao \footnote{South China Research Center for Applied Mathematics and Interdisciplinary Studies, South China Normal University, Guangzhou 510631, China.  (zhwliao@hotmail.com)}, Jinghai Shao \thanks{Center for Applied Mathematics, Tianjin University, Tianjin 300072, China. (shaojh@tju.edu.cn.)}}
\maketitle

\begin{abstract}
  In this work we investigate the long time behavior of the Ornstein-Uhlenbeck process driven by L\'evy noise with regime-switching. We provide explicit criteria on the transience and recurrence of this process. Contrasted with the Ornstein-Uhlenbeck process driven simply by Brownian motion, whose stationary distribution must be light-tailed, both the jumps caused by the L\'evy noise and regime-switching described by Markov chain can derive the heavy-tailed property of the stationary distribution. In this work, the different role played by L\'evy measure and   regime-switching process is clearly characterized.
\end{abstract}



\section{Introduction} \label{sect1}

In this work we are concerned with the existence of stationary distributions of regime-switching processes driven by L\'evy noises, and further analyze the tail behavior of these distributions.
This work is motivated by the works of de Saporta and Yao \cite{DY} and Bardet et al. \cite{BGM10}. It has been shown in different ways that in contrast with the classical Ornstein-Uhlenbeck process whose stationary distribution must be light-tailed, the stationary distribution of the Ornstein-Uhlenbeck process with regime-switching  may have a heavy-tailed stationary distribution. Moreover, in a recent work \cite{HS17}, Hou and Shao showed that the Cox-Ingersoll-Ross (CIR) process with regime-switching may have a heavy-tailed stationary distribution, which provides a suitable modification of the classical CIR model due to its poor empirical performance shown by Brown and Dybvig \cite{BD}. Hence, from the  application viewpoint, it is quite necessary to study the stochastic processes with regime-switching.

The L\'evy-driven Ornstein-Uhlenbeck process has a lot of applications in the study of mathematical finance (cf. e.g. \cite{BMR01, BS01, CPY, LM05, Nov03} amongst others).
In this work, we shall extend the study of \cite{DY} and \cite{BGM10} to the L\'evy-driven Ornstein-Uhlenbeck process with regime-switching. We shall provide explicit conditions to justify the recurrence and transience of such process, and find out the role played by random change of the environment. Furthermore, provided the existence of stationary distribution, we provide explicit conditions to justify whether it is light- or heavy-tailed. In this part, we are mainly interested in finding out the different role played by L\'evy jumps and regime-switching in determining the heavy-tailed property of the stationary distribution.

Precisely, we consider the following stochastic processes $(X_t,\La_t)$:
\begin{equation}\label{1.1}
\d X_t=\alpha_{\La_t}X_t\d t+\sigma_{\La_t}\d Z_t,\qquad X_0=x_0\in \R,
\end{equation}
and $(\La_t)$ is a continuous-time Markov chain on a finite state space $\S=\{1,2,\ldots,N\}$, $2\leq N<\infty$, with the transition rate matrix $Q=(q_{ij})_{i,j\in \S}$, which is assumed to be conservative and irreducible. Here $\alpha:\S\to \R$, $\sigma:\S\to \R$ are measurable functions, and $Z=(Z_t)_{t\geq 0}$ is a L\'evy processes on $\R$ given by
\begin{equation}\label{1.2}
Z_t=b t+\sqrt{a} B_t+\int_0^t\int_{0<|z|<1} z\wt N(\d t,\d z)+\int_0^t \int_{|z|\geq 1} zN(\d t,\d z),
\end{equation}
where $b\in \R$, $a>0$, $B=(B_t)_{t\geq 0}$ is a standard 1-dimension Brownian motion, $N(\d t,\d z)$ is a Poisson random measure on $[0,\infty)\times \R\backslash\{0\}$, and $\wt N(\d t,\d z)$ is its corresponding compensated Poisson random measure relative to the L\'evy measure $\nu$, i.e. $\wt N(\d t,\d z)=N(\d t,\d z)-\nu(\d z)\d t$. Throughout this work, it is assumed that $(\La_t)$ is independent of the L\'evy process $(Z_t)$. As a L\'evy measure, $\nu$ is a $\sigma$-finite measure on $\R$ satisfying $\nu(\{0\})=0$ and
\begin{equation}\label{1.3}
\int_{z\neq 0}\big(1\wedge |z|^2\big)\,\nu(\d z)<\infty.
\end{equation}

Since the Markov chain $(\La_t)$ is in a finite state space with a conservative and  irreducible $Q$-matrix, there exists a unique stationary distribution, denoted by $\mu=(\mu_i)_{i\in \S}$.
Our first result is on the recurrent property of $(X_t,\La_t)$.
\begin{mythm}\label{t-1}
\begin{itemize}
\item[$(i)$] Assume that
\begin{equation}\label{a-1}
\int_{|z|\neq 0} \log\big(1+|z|\big)\nu(\d z)<\infty,
\end{equation}
then $(X_t,\La_t)$ is positive recurrent if $\sum_{i\in\S} \mu_i \alpha_i<0$.
\item[$(ii)$] Assume that
\begin{equation}\label{a-1.5}
\int_{|z|\neq 0} |z|\vee |z|^2\nu(\d z)<\infty,
\end{equation} then $(X_t,\La_t)$ is transient if $\sum_{i\in \S} \mu_i\alpha_i>0$.
\end{itemize}
\end{mythm}

Restricted to the set $\{|z|\geq 1\}$, our condition \eqref{a-1} is just the condition (1.3) in \cite{SY84}. In \cite{SY84}, Sato and Yamazato showed that this condition is a sharp condition on the integrability of L\'evy measure $\nu$ so that the associated Ornstein-Uhlenbeck process without regime-switching to be positive recurrent. Later  Shiga \cite{Shi} provided a recurrence criterion for this kind of process in one dimension and discussed several symmetric multidimensional cases.    \cite{SWY} and \cite{SWYY} also investigated recurrence criterion in multidimensional case, and finally Watanabe \cite{Wat} used the Fourier analytic method to provide a general recurrence criterion for L\'evy-driven Ornstein-Uhlenbeck processes, which solved Sato's conjecture in the affirmative.

Next, when $(X_t,\La_t)$ is positive recurrent, it owns a unique stationary distribution $\pi$. Concerning the tail behavior of $\pi$, we establish the following result:
\begin{mythm}\label{t-2}
Assume that \eqref{a-1} and $\sum_{i\in\S}\mu_i\alpha_i<0$ hold, and denote by $\pi$ the stationary distribution of $(X_t,\La_t)$.
\begin{itemize}
  \item[$(i)$] If there exists some $i_0\in\S$ such that for any $\lambda>0$,
  \begin{equation}\label{a-2}
  \int_{|z|\geq 1} \Big(\e^{\lambda \sigma_{i_0} z}-1\Big)\nu(\d z)=\infty,
  \end{equation}
  then $\pi$ is heavy-tailed, i.e.
  \begin{equation}\label{a-3}
   \sum_{j\in \S}\int_{\R}\e^{\lambda |x|} \pi(\d x, j)=\infty,\quad \forall\, \lambda>0.
  \end{equation}
  \item[$(ii)$] Suppose there exists a $\lambda_0>0$ such that for each $i\in\S$,
  \begin{equation}\label{a-4}
  \int_{|z|\geq 1} \Big(\e^{\lambda_0|\sigma_i z|}-1\Big)\nu(\d z)<\infty.
  \end{equation}
   If $\max_{i\in\S}\alpha_i<0$, then $\pi$ is light-tailed. If $\max_{i\in\S}\alpha_i>0$, then
   for any $p>\kappa \vee 2$,
   \[\sum_{j\in \S}\int_{\R}|x|^p\pi(\d x,j )=\infty,\]
   which means that $\pi$ is heavy-tailed, where $\kappa>0$ is given in \eqref{kap} below.
\end{itemize}
\end{mythm}

For any $p>0$, define $Q_p=Q+p\mathrm{diag}(\alpha_1,\ldots,\alpha_N)$, where $Q=(q_{ij})$ is the transition rate matrix of $(\La_t)$, and $\mathrm{diag}(\alpha_1,\ldots,\alpha_N)$ is the diagnal matrix generated by the vector $(\alpha_1,\ldots,\alpha_N)$.
Let
\begin{equation}\label{eta}
\eta_p=-\max_{\gamma\in \mathrm{Spec}(Q_p)}\mathrm{Re}\,\gamma,
\end{equation}
where $\mathrm{Spec}(Q_p)$ stands for the spectrum of $Q_p$.
Define
\begin{equation}\label{kap}
\kappa=\sup\{p>0; \eta_p>0\}.
\end{equation}
According to \cite[Propositions 4.1, 4.2]{BGM10}, $\kappa=\infty$ if $\max_{i\in \S}\leq 0$; otherwise, $\kappa\in (0,\min\{q_i/\alpha_i;\alpha_i>0\})$. Moreover, for $p>\kappa$, $\eta_p<0$ and for $p\in (0,\kappa)$, $\eta_p>0$.

\begin{myrem}
  According to the characterization of the tail behavior for L\'evy-driven Ornstein-Uhlenbeck process in Lemma \ref{t-4} below, assertion (i) of Theorem \ref{t-2} tells us that if in some fixed environment, the L\'evy measure $\nu$ makes the stationary distribution for the Ornstein-Uhlenbeck process without switching to be heavy-tailed, then the stationary distribution of Ornstein-Uhlenbeck process with switching must be heavy-tailed regardless of the random switching of the environment.

  Assertion (ii) of Theorem \ref{t-2} tells us that under a little stronger condition on $\nu$ such that the Ornstein-Uhlenbeck process in every fixed environment must own light-tailed stationary distribution provided it exists, then the random switching $(\La_t)$ can change the tail behavior of the stationary distribution of $(X_t,\La_t)$ according to the signal of $\max_{i\in\S}\alpha_i$.
\end{myrem}

The proofs of  Theorem \ref{t-1} and Theorem \ref{t-2} are presented in the next section.

\section{Proofs of the main results}

In this section, we shall first present the argument of Theorem \ref{t-1} on the recurrent property of the process $(X_t,\La_t)$. To this end, we need to extend the Lyapunov-type criterion established by Shao in \cite{Sh15} for stochastic processes driven by the Brownian motion to the current situation.

Let us first introduce some notations.
For each $i\in \S$,  define
\begin{equation}\label{b-1}
\begin{split}
\D^{(i)}f(\cdot,i)(x)&= (\alpha_i x+b\sigma_i)\frac{\partial f(x,i)}{\partial x}+\frac{a\sigma_i^2}{2}\frac{\partial^2f(x,i)}{\partial x^2},\\
\J^{(i)} f(\cdot,i)(x)&=\int_{z\neq 0}\!\Big[f(x\!+\!\sigma_iz,i)\!-\!f(x,i)\!-\!\sigma_i z\frac{\partial f(x,i)}{\partial x}\mathbf{1}_{0<|z|<1}\Big]\nu(\d z),\\
  L^{(i)} f(\cdot,i)(x)&=\mathcal{D}^{(i)} f(\cdot,i)(x)+\Ji f(\cdot,i)(x),
\end{split}
\end{equation}
for $f(\cdot,i)\in C^2(\R)$. Then the infinitesimal generator of the process $(X_t,\La_t)$ is given by \begin{equation}\label{b-2}
\begin{split}
\A f(x,i)&=L^{(i)} f(\cdot,i)(x)+Qf(x,\cdot)(i)\\
&=L^{(i)}f(\cdot,i)(x)+\sum_{j\neq i}q_{ij}(f(x,j)-f(x,i)),\quad f(\cdot,i)\in C^2(\R).
\end{split}
\end{equation}

Following the same approach as \cite[Theorem 2.1, Theorem 3.1]{Sh15}, we can establish the following criterion to justify the recurrence of L\'evy-driven process $(X_t,\La_t)$. The details are omitted.
\begin{myprop}\label{t-3} Let $\mu$ be the stationary distribution of $(\La_t)$.

(i) Suppose that there exist constants $r_0>0$, $\beta_i\in \R$, $i\in\S$, positive function $V(x)$ such that
\begin{equation}\label{ly-1}
L^{(i)} V(x)\leq \beta_iV(x),\quad |x|\geq r_0, \quad i\in\S,
\end{equation}
satisfying $\sum_{i\in\S}\mu_i \beta_i<0$.
Then $(X_t,\La_t)$ is exponentially ergodic if $\lim_{|x|\to\infty} V(x)=\infty$, and is transient if $\lim_{|x|\ra \infty} V(x)=0$.

(ii) Suppose that there exist constants $r_0>0$,  $\beta_i\in \R$, $i\in \S$,  two positive functions $h,\,g \in C^2(\R)$    such that for each $i\in \S$,
\begin{equation}\label{b-6}
 L^{(i)}h(x)\leq \beta_i g(x),\quad \ |x|>r_0,
\end{equation}
and
\begin{equation}\label{b-7}
\lim_{|x|\to\infty} \frac{g(x)}{h(x)}=0,\quad \lim_{|x|\to \infty} \frac{L^{(i)} g(x)}{g(x)}=0.
\end{equation}
Assume that
\begin{equation}\label{b-8}
\sum_{i\in\S}\mu_i\beta_i<0.
\end{equation}
Then $(X_t,\La_t)$ given by \eqref{1.1} and \eqref{1.2} is positive recurrent if $\lim_{|x|\to \infty} h(x)=\infty$, and is transient if $\lim_{|x|\to \infty} h(x)=0$.
\end{myprop}
\begin{myrem}
  In this proposition \ref{t-3}, we state that $(X_t,\La_t)$ is positive recurrent when $\sum_{i\in\S}\mu_i\beta_i<0$ and $\lim_{|x|\to\infty} h(x)=\infty$, which is a little stronger than the result presented in \cite[Theorem 3.1]{Sh15}. However, the argument in \cite{Sh15} is sufficient to this statement according to the Foster-Lyapunov criterion.
\end{myrem}
\noindent\textbf{Proof of Theorem \ref{t-1}}\\
First, let us prove the positive recurrence of $(X_t,\La_t)$ in the situation $\sum_{i\in\S}\mu_i\alpha_i<0$ by part (ii) of Proposition 2.1. Consider the following auxiliary functions
\[h(x)=\log(1+x^2),\quad g(x)=\frac{2x^2}{1+x^2}.\]
It is easy to check that $\dis\lim_{|x|\to \infty} h(x)=\infty$ and $\dis\lim_{|x|\to \infty} \nicefrac{g(x)}{h(x)}=0$.

Now we verify that $\lim_{|x|\to\infty} \frac{L^{(i)}g(x)}{g(x)}=0$.
Note that
\begin{equation}\label{c-1}
\frac{\Di g(x)}{g(x)}=\frac{2\alpha_i}{1+x^2}+\frac{2b\sigma_i}{x(1+x^2)}+\frac{a\sigma_i^2(1-3x^2)}{x^2(1+x^2)^2}
\longrightarrow 0,\quad \text{as}\ |x|\to \infty.
\end{equation}
For the jumping component, it holds
\begin{equation}\label{c-2}
\begin{split}
\frac{\Ji g(x)}{g(x)}&=\Big(1+\frac 1{x^2}\Big)\int_{z\neq 0}\!\frac{\sigma_i^2z^2+2\sigma_i xz}{(1+x^2)(1+(x+\sigma_iz)^2)}\nu(\d z)\\
&\qquad \qquad -\frac{2\sigma_i}{x(1+x^2)}\int_{z\neq 0}\!\big(1\wedge |z|\big)\nu(\d z).
\end{split}
\end{equation}
Due to  condition \eqref{a-1},
it is obvious that
\[\lim_{|x|\to \infty} \frac{2\sigma_i}{x(1+x^2)}\int_{z\neq 0}\!\big(1\wedge |z|\big)\nu(\d z)=0.
\]
Because
\[\Big|\frac{\sigma_i^2z^2+2\sigma_i xz}{(1+x^2)(1+(x+\sigma_iz)^2)}\Big|=\Big|\frac{(\sigma_iz+x)^2-x^2}{(1+x^2)(1+(x+\sigma_iz)^2)}\Big|\leq 1,\]
and
\[\Big|\frac{\sigma_i^2z^2+2\sigma_i xz}{(1+x^2)(1+(x+\sigma_iz)^2)}\Big|\leq \Big|\frac{\sigma_i^2 z^2+2\sigma_ixz}{(1+x^2)}\Big|\leq \sigma_i^2 z^2+|\sigma_iz|,\]
we obtain that
\[\Big|\frac{\sigma_i^2z^2+2\sigma_i xz}{(1+x^2)(1+(x+\sigma_iz)^2)}\Big|\leq \min\big\{1, \sigma_i^2z^2+|\sigma_i z|\big\}.
\]
Invoking condition \eqref{a-1}, the dominated convergence theorem yields that
\begin{equation}\label{c-3}
\lim_{|x|\to \infty} \Big(1+\frac{1}{x^2}\Big)\int_{z\neq 0}\!\frac{\sigma_i^2+2\sigma_i xz}{(1+x^2)(1+(x+\sigma_iz)^2)}\nu(\d z)=0.
\end{equation}
Hence,
\begin{equation}\label{c-4}
\lim_{|x|\to \infty} \frac{\Ji g(x)}{g(x)}=0.
\end{equation} Combining with \eqref{c-1}, this further implies
\[\lim_{|x|\to \infty} \frac{L^{(i)} g(x)}{g(x)}= 0.\]

Next, we shall show that for any $\veps$ satisfying $-\sum_{i\in\S}\mu_i\alpha_i>\veps>0$, there exists a constant $r_0>0$ such that for every $i\in\S$,
\begin{equation}\label{c-5}
L^{(i)} h(x)\leq (\alpha_i+\veps) g(x),\qquad \text{as}\ |x|\geq r_0.
\end{equation}
So, by taking $\beta_i=\alpha_i+\veps$ in \eqref{b-6}, Proposition \ref{t-3} yields that
$(X_t,\La_t)$ is positive recurrent.

Indeed, for the diffusion part of $L^{(i)}h$, we have
\begin{equation}\label{c-6}
\begin{split}
  \Di h(x)&=(\alpha_i x+b\sigma_i)\frac{2x}{1+x^2}+a\sigma_i^2\frac{1-x^2}{(1+x^2)^2}\\
  &\leq \Big(\alpha_i+\frac{b\sigma_i}x+\frac{a\sigma_i^2|1-x^2|}{2x^2(1+x^2)}\Big)g(x).
  \end{split}
\end{equation} Due the finiteness of the number of states in $\S$, it is clear that there exists $r_1>0$ such that
\begin{equation}\label{c-7}\frac{b\sigma_i}{x}+\frac{a\sigma_i^2|1-x^2|}{2x^2(1+x^2)}\leq \frac{\veps}{4}, \quad \forall\,|x|\geq r_1,\ i\in \S.
\end{equation}

For the jump part of $L^{(i)} h$, direct calculation leads to
\begin{align*}
  \Ji h(x)&=\int_{z\neq 0}\!\big[\log(1+(x+\sigma_i z)^2)-\log(1+x^2)\big]\nu(\d z)-\frac{2\sigma_ix}{1+x^2}\int_{0<|z|<1}\!z\,\nu(\d z).
\end{align*}
Since
\begin{align*}
  \log(1+(x+\sigma_iz)^2)-\log(1+x^2)&\leq \log\Big(1+\frac{\sigma_i^2z^2}{1+x^2}+\frac{2x}{1+x^2}|\sigma_iz|\Big)\\
  &\leq \log(1+|\sigma_i z|+\sigma_i^2z^2)\leq 2\log(1+|\sigma_iz|),
\end{align*}
we obtain from condition \eqref{a-1} and the dominated convergence theorem that
\[\lim_{|x|\to \infty}  \Ji h(x) =0.\]
Noting that $g(x)\geq 1$ for $|x|\geq 1$, there exists $r_2>1$ such that for any $|x|\geq r_2$,
\begin{equation}\label{c-8}
\Ji h(x)\leq \frac{\veps}{4} g(x).
\end{equation}
Consequently, \eqref{c-5} follows from \eqref{c-6}, \eqref{c-7} and \eqref{c-8} by taking $  r_0=\max\{r_1,r_2\}$.

Second, we go to prove the transience of $(X_t,\La_t)$ in the case  $\sum_{i\in\S}\mu_i\alpha_i>0$.
To this aim, we consider the function
\[V(x)=\frac{1}{\delta+x^2}\]
for some fixed $0 < \delta < 1$, and prove the desired result according to part (i) of Proposition \ref{t-3}. It holds
\begin{equation}\label{d-1}
\Di V(x)=\Big(-2\alpha_i\frac{x^2}{\delta+x^2}
-2b\sigma_i\frac{x}{\delta+x^2}+a\sigma_i^2\frac{3x^2-\delta}{(\delta+x^2)^2}\Big)V(x),
\end{equation}
and clearly
\[\lim_{|x|\to \infty} -2\alpha_i\frac{x^2}{\delta+x^2}
-2b\sigma_i\frac{x}{\delta+x^2}+a\sigma_i^2\frac{3x^2-\delta}{(\delta+x^2)^2}=-2\alpha_i.
\]
For the jump part of $L^{(i)}V$, it holds
\begin{equation}\label{d-2}
\Ji V(x)=\Big( \int_{z\neq 0}\!\!\frac{x^2-(x+\sigma_i z)^2}{\delta+(x+\sigma_iz)^2}\nu(\d z)+\frac{2\sigma_i x}{\delta+x^2}\int_{0<|z|<1}\!\! z\nu(\d z)\Big) V(x).
\end{equation}
Moreover,
\begin{align*}
\int_{z\neq 0}\frac{x^2-(x+\sigma_iz)^2}{\delta+(x+\sigma_iz)^2}\nu(\d z)&=\int_{|z|\geq 1}\!\frac{x^2-(x+\sigma_iz)^2}{\delta+(x+\sigma_iz)^2}\nu(\d z)+\int_{0<|z|<1}\!\frac{x^2-(x+\sigma_iz)^2}{\delta+(x+\sigma_iz)^2}\nu(\d z)\\
&=:I_1+I_2.
\end{align*}
To estimate the term $I_1$, since $0 < \delta < 1$, we have
\begin{align*}
     \Big|\frac{x^2-(x+\sigma_iz)^2}{\delta+(x+\sigma_iz)^2}\Big|
  &\leq   \frac{2|x\sigma_i z|}{\delta+(x+\sigma_iz)^2}+\frac{\sigma_i^2 z^2}{\delta} \\
  &\leq  \frac{(x+\sigma_iz)^2}{\delta+(x+\sigma_iz)^2}+\frac{\sigma_i^2 z^2}{\delta}\\
   &\leq   \frac{1}{\delta}+\frac{\sigma_i^2}{\delta}z^2,
\end{align*}
and $\frac{1}{\delta}+\frac{\sigma_i^2z^2}{\delta}$ is an integrable function with respect to $\mathbf{1}_{|z|\geq 1}\nu(\d z)$  by condition \eqref{a-1.5}. Then the dominated convergence theorem leads to
\begin{equation}\label{d-3}
\lim_{|x|\to \infty} \int_{|z|\geq 1} \frac{x^2-(x+\sigma_iz)^2}{\delta+(x+\sigma_iz)^2}\nu(\d z)  =0.
\end{equation}

To estimate the term $I_2$, let us consider
\[u(x):=\frac{2x\sigma_iz}{\delta+(x+\sigma_iz)^2} \quad \text{for $0<|z|<1$},\]
and direct calculation leads to
\[u'(x)=\frac{2\sigma_iz \delta-2\sigma_izx^2+2(\sigma_i z)^3}{(\delta+(x+\sigma_iz)^2)^2}.\]
Hence, there exists $r_1>0$ such that for $|x|\geq r_1$, $u'(x)>0$ if $\sigma_iz<0$ and $u'(x)<0$ if $\sigma_iz>0$, which means that $|u(x)|$ is monotone when $|x|\geq r_1$.
Noting that
$\lim_{|x|\to\infty} |u(x)|=0$, we have
\begin{equation}\label{d-4}
\sup_{|x|\geq r_1} |u(x)|=\sup_{|x|=r_1} |u(x)|=\max\Big\{\frac{2r_1 \sigma_i z}{\delta+(r_1+\sigma_iz)^2}, \frac{-2r_1\sigma_iz}{\delta+(\sigma_iz-r_1)^2}\Big\},
\end{equation}
Since
\[ \Big|\frac{x^2-(x+\sigma_iz)^2}{\delta+(x+\sigma_iz)^2}\Big|\leq |u(x)|+\frac{\sigma_i^2 z^2}{\delta},\]
it follows from \eqref{d-4}, condition \eqref{a-1.5} and the dominated convergence theorem that
\begin{equation}\label{d-5}
\lim_{|x|\to \infty} \int_{0<|z|<1}\!\frac{x^2-(x+\sigma_iz)^2}{\delta+(x+\sigma_iz)^2}\nu(\d z)=0.
\end{equation}

Invoking \eqref{d-1}, \eqref{d-3}, \eqref{d-5} and condition \eqref{a-1}, by \eqref{d-2}, for any $\veps$ in $(0,\sum_{i\in\S} \mu_i\alpha_i)$, there exists $r_0>r_1>0$ such that
\begin{equation}\label{d-6}
 L^{(i)} V(x)=\Di V(x)+\Ji V(x)\leq \big(-2\alpha_i+ \veps\big)V(x),\qquad |x|\geq r_0.
\end{equation}
According to Proposition \ref{t-3}, $(X_t,\La_t)$ is transient.
The proof is completed. \fin

Next, we proceed to investigate the tail property of the stationary distribution $\pi(\d x, i)$ of $(X_t,\La_t)$.
Before this, we present a lemma on the tail property of the stationary distribution of  L\'evy-driven Ornstein-Uhlenbeck process without regime-switching, which tells us how the jumping component impacts the tail behavior of its stationary distribution.

Consider the process
\begin{equation}\label{e-1}
\d Y_t=\alpha Y_t\d t+\sigma \d Z_t, \quad Y_0=y_0\in \R,
\end{equation}
where $\alpha<0$, $\sigma\in \R$ and $(Z_t)_{t\geq 0}$ is still given by \eqref{1.2} satisfying condition \eqref{a-1}. Then the process $(Y_t)$ is positive recurrent (cf. Theorem \ref{t-1} or \cite[Theorem 4.1]{SY84}), and its stationary distribution is denoted by $\tilde \pi$.

\begin{mylem}\label{t-4}
Assume condition \eqref{a-1} holds. Consider the process $(Y_t)$ given by \eqref{e-1} and its stationary distribution $\tilde \pi$.
\begin{itemize}
  \item[$(i)$] If for some $\lambda_0>0$ such that
  \begin{equation}\label{e-2}
  \int_{|z|\geq 1} \Big(\e^{\lambda_0 \sigma z}-1\Big)\nu(\d z)<\infty,
  \end{equation}
  then $\int_{\R} \e^{\lambda_0 |x|}\tilde \pi(\d x)<\infty$, which means that $\tilde \pi$ is light-tailed.
  \item[$(ii)$] If for any $\lambda>0$,
  \begin{equation}\label{e-3}
  \int_{|z|\geq 1}\Big(\e^{\lambda\sigma z}-1\Big)\nu(\d z)=\infty,
  \end{equation} then for any $\lambda>0$, $\int_{\R} \e^{\lambda x}\tilde \pi(\d x)=\infty$, which means that $\tilde \pi$ is heavy-tailed.
\end{itemize}
\end{mylem}

\begin{proof}
(i)\ The characteristic function of the L\'evy process $(Z_t)$ given by \eqref{1.2} is
\[\E\Big(\e^{\ii u Z_t}\Big)=\e^{-t\Phi(u)},\quad u\in \R, \ t\geq 0,\]
 where $\ii$ is the imaginary unit, and $\Phi$ is given by the L\'evy-Khintchine representation
   \begin{equation}\label{e-4}
   \Phi(u)=\frac12 au^2-\ii bu+\int_{z\neq 0}\!\big(1-\e^{\ii uz}+\ii uz\mathbf{1}_{0<|z|<1}\big)\nu(\d z).
   \end{equation}
 The stationary distribution $\tilde \pi$ of $(Y_t)$ has the following characteristic function
   \begin{equation}\label{e-5}
   \widehat{\tilde\pi}(z)=\exp\Big(\int_0^\infty -\Phi\big(\e^{\alpha t}\sigma z\big)\d t\Big).
   \end{equation}
   This yields
   \begin{equation}\label{e-6}
   \int_{\R} \e^{\lambda_0 x}\tilde \pi(\d x)=\widehat{\tilde \pi}(-\ii \lambda_0)=\exp\Big(\int_0^\infty -\Phi\big(-\ii \lambda_0 \sigma \e^{\alpha t}\big)\d t\Big).
   \end{equation}
Due to \eqref{e-4},
   \begin{equation}\label{e-7}
   \begin{split}
   \int_0^\infty \!\!\!-\Phi\big(-\ii \lambda_0 \sigma \e^{\alpha t}\big)\d t&=\int_0^\infty \Big(\frac 12a \lambda_0^2\sigma^2\e^{2\alpha t}\!+\!b\lambda_0 \sigma \e^{\alpha t}\Big)\d t\\
   &\quad+\int_0^\infty\!\!\int_{z\neq 0}\!\big(\exp\big(z\lambda_0 \sigma \e^{\alpha t}\big)\!-\!1\!-\!z\lambda_0 \sigma \e^{\alpha t}\mathbf{1}_{0<|z|<1}\big)\nu(\d z)\d t\\
   =:I_1+I_2.
   \end{split}
   \end{equation}
As $\alpha<0$, it is clear that
\[I_1=-\frac{a\lambda_0^2\sigma^2}{4\alpha}-\frac{b\lambda_0 \sigma}{\alpha}<\infty.\]
Performing the variable substitution $u=\e^{\alpha t}$, we get
\begin{align*}I_2&=-\frac{1}{\alpha}\int_0^1\int_{z\neq 0}\!\frac{1}{u}\Big(\e^{z\lambda_0 \sigma u}-1\Big)\nu(\d z)\d u+\frac{\lambda_0 \sigma}{\alpha}\Big(\int_{z\neq 0}\!z\mathbf1_{0<|z|<1}\nu(\d z)\Big)\\
&\leq -\frac{1}{\alpha}\int_{z\neq 0}\Big(\e^{\lambda_0\sigma z}-1\Big)\nu(\d z)+\frac{\lambda_0 \sigma}{\alpha}\Big(\int_{z\neq 0}\!z\mathbf1_{0<|z|<1}\nu(\d z)\Big),
\end{align*}
where we  used the fact the function $f (x):=(\e^{c x}-1)/x $ is increasing for any $c\neq 0$.
Combining these estimates on $I_1$, $I_2$ with \eqref{e-6}, \eqref{e-7}, by \eqref{a-1} and \eqref{e-2}, we obtain finally that
\[\int_{\R} \e^{\lambda_0 x}\tilde \pi(\d x)<\infty.
\]
Similarly, we can show
$\int_{\R}\e^{-\lambda_0 x}\tilde \pi(\d x)<\infty$.
Noting that $\e^{\lambda_0|x|}\leq \e^{\lambda_0 x}+\e^{-\lambda_0 x}$, we obtain that
\[\int_{\R}\e^{\lambda_0 |x|}\tilde \pi(\d x)<\infty.\]

(ii)\ Replacing $\lambda_0$ with any positive constant $\lambda$ in \eqref{e-6} and \eqref{e-7}, we still have that
\[I_1=-\frac{a\lambda^2\sigma^2}{4\alpha}-\frac{b\lambda \sigma}{\alpha}<\infty\]
since $\alpha<0$. However, in current case,
\begin{equation}\label{e-8}
I_2=-\frac{1}{\alpha}\int_0^1\!\!\int_{z\neq 0}\frac1{u}\Big(\e^{z\lambda\sigma u}-1\Big)\nu(\d z)\d u+\frac{\lambda\sigma}{\alpha}\Big(\int_{z\neq 0}\!\!z\mathbf{1}_{0<|z|<1}\nu(\d z)\Big).
\end{equation}
By conditions \eqref{a-1} and \eqref{e-3}, there exists some $C>0$ such that
\[\int_{0<|z|<1} \frac1u\Big(\e^{z\lambda \sigma u}-1\Big)\nu(\d z)\leq C\int_{0<|z|<1} |z|\nu(\d z)<\infty,\]
and
\[\int_{|z|\geq 1}\! \frac1u\Big(\e^{z\lambda \sigma u}-1\Big)\nu(\d z)=\infty,\quad u\neq 0.\]
As a consequence,
\[\int_{z\neq 0}\frac1{u}\Big(\e^{z\lambda\sigma u}-1\Big)\nu(\d z)=\infty,\]
which implies $I_2=\infty$ and
\[\int_{\R}\!\e^{\lambda x}\tilde \pi(\d x)=\infty.\]
The arbitrariness of the positive constant $\lambda$ implies that $\tilde \pi$ is heavy-tailed.
The proof of this lemma is complete.
\end{proof}

\noindent\textbf{Proof of Theorem \ref{t-2}}\\
(i)\ We prove this assertion by contradiction. Assume that there exists a $\delta_0>0$ such that
\begin{equation}\label{f-1}
\sum_{j\in\S}\int_{\R}\e^{2\delta_0|x|}\pi(\d x, j)<\infty.
\end{equation}
Applying It\^o's formula, we obtain that
\begin{equation}\label{f-2}
\begin{split}
  &\E\e^{\delta_0 X_t}-\E\e^{\delta_0 X_0}\\
  &=\E\int_0^t\delta_0\e^{\delta_0 X_{s-}} \big(\alpha_{\La_{s-}}X_{s-}+b\sigma_{\La_{s-}}+\frac{a\delta_0}{2}\sigma_{\La_{s-}}^2\big)\d s\\
  &\quad+2\E\int_0^t\int_{|z|\geq 1} \e^{\delta_0 X_{s-}}\big(\e^{\delta_0\sigma_{\La_{s-}} z}-1\big)\nu(\d z)\d s\\
  &\quad+\E\int_0^t\int_{0<|z|<1}\!\!\e^{\delta_0 X_{s-}}\big(\e^{\delta_0\sigma_{\La_{s-}}z}-1-\delta_0\sigma_{\La_{s-}}z \big)\nu(\d z)\d s\\
  &=:I_1+I_2+I_3
\end{split}
\end{equation}
If we take the initial distribution of $(X_0,\La_0)$ to be the stationary distribution $\pi$, then for every $t>0$, the distribution of $(X_t,\La_t)$ remains to be $\pi$, which implies immediately that the left-hand side of \eqref{f-2} equals to $0$.
From $\sum_{j\in \S} \int_{\R}\e^{2\delta_0 |x|}\pi(\d x,j)<\infty$ and the finiteness of the number of states in $\S$, it is easy to see that
\[\E\big[ \e^{\delta_0 X_{s-}}\big( |\alpha_{\La_{s-}}X_{s-}|
+|b\sigma_{\La_{s-}}|+\frac{a\delta_0}{2}\sigma_{\La_{s-}}^2\big)\big]<\infty,\]
and hence the term $I_1<\infty$.
Condition \eqref{a-1} ensures that there exists a constant $C>0$ such that for every $j\in \S$,
\[\int_{0<|z|<1} \big|\e^{\delta_0\sigma_j z}-1-\delta_0 \sigma_jz\big|\nu(\d z)\leq C\int_{0<|z|<1} |z|\nu(\d z)<\infty.\]
Then using again \eqref{f-1}, we have $I_3<\infty$.
Nevertheless, the term $I_2$ is equal to $\infty$. Indeed, when $\La_{s-}=i_0$, it follows from condition \eqref{a-2} that
\[ \int_{|z|\geq 1}\big(\e^{\delta_0\sigma_{\La_{s-}}z}-1\big)\nu(\d z)=\infty.\]
As an irreducible Markov chain, $(\La_t)$ arrives at the state $i_0$ with a positive probability, so $I_2=\infty$. Therefore, we get a contradiction that the left-hand side of \eqref{f-2} equals to 0, but the right-hand side of \eqref{f-2} equals to $\infty$. Consequently, $\pi$ cannot be light-tailed in current situation.

(ii) First, we go to show $\pi$ is light-tailed under conditions \eqref{a-4} and $\max_{i\in \S} \alpha_i<0$. For $\delta\in (0,\lambda_0)$, similar to \eqref{f-2}, It\^o's formula yields that
\begin{equation}\label{f-3}
\begin{split}
  &\E \e^{\delta |X_t|}-\E\e^{\delta |X_0|}\\
  &=\E\int_0^t\Big[\delta\e^{\delta  |X_{s-}|}\sgn(X_{s-})\big(\alpha_{\La_{s-}}X_{s-}+b\sigma_{\La_{s-}})+
   \frac{a\delta^2}{2}\e^{\delta |X_{s-}|}\sigma_{\La_{s-}}^2\Big]\d s\\
  &\quad+2\E\int_0^t\int_{|z|\geq 1} \big(\e^{\delta  |X_{s-}+\sigma_{\La_{s-}} z|}-\e^{\delta  |X_{s-}|}\big)\nu(\d z)\d s\\
  &\quad+\E\int_0^t\int_{0<|z|<1}\!\!\big(\e^{\delta| X_{s-}\!+\!\sigma_{\La_{s-}}z|}\!-\!\e^{\delta |X_{s-}|}
  \!-\!\delta \e^{\delta |X_{s-}|}\sgn(X_{s-})\sigma_{\La_{s-}}z\big)\nu(\d z)\d s\\
  &=:J_1+J_2+J_3,
\end{split}
\end{equation}
where $\sgn(x)$ denotes the sign function of $x$.

Since $\max_{i\in \S}\alpha_i<0$, then for any $c>0$ there exists $M>0$ such that
\begin{equation}\label{ff}
\alpha_{\La_{s-}}|X_{s-}|+\sgn(X_{s-})b\sigma_{\La_{s-}}\leq -c+M \e^{-\delta |X_{s-}|}.
\end{equation}
Therefore,
\begin{equation}\label{f-4}
J_1\leq \int_0^t \E\big(  \delta M -c\delta \e^{\delta |X_{s-}|} \big) \d s
\end{equation}

For $J_2$, we have
\begin{equation}\label{f-5}
\begin{split}
J_2&\leq \E\int_0^t\e^{\delta |X_{s-}|}\Big[\int_{|z|\geq 1}\big(\e^{\delta|\sigma_{\La_{s-}}z|}-1\big)\nu(\d z)\Big]\d s\\
&\leq K_1(\delta)\int_0^t \E\big[\e^{\delta |X_{s-}|}\big]\d s,
\end{split}
\end{equation}
where
\[K_1(\delta)=\max_{i\in\S} \int_{|z|\geq 1} \big(\e^{\delta|\sigma_i z|}-1\big)\nu(\d z)<\infty, \]
due to \eqref{a-4} and $\delta<\lambda_0$. Moreover,
\[\lim_{\delta\to 0}\frac{K_1(\delta)}{\delta}\leq \lim_{\delta\to 0}\int_{|z|\geq 1}\!\!\frac{1}{\delta}\big(\e^{\delta\sigma_\ast |z|}-1\big)\nu(\d z)
=\int_{|z|\geq 1}\!\!\sigma_\ast |z|\nu(\d z)<\infty,\]
where $\sigma_\ast :=\max_{i\in \S} |\sigma_i|$.

For $J_3$, it holds
\begin{equation}\label{f-6}
\begin{split}
  J_3&\leq \E\int_0^t\!\int_{0<|z|<1}\!\!\e^{\delta |X_{s-}|}\Big(\e^{\delta|\sigma_{\La_{s-}}z|}-1+\delta|\sigma_{\La_{s-}}z|\Big)\nu(\d z)\d s\\
  &\leq K_2(\delta)\int_0^t\E\e^{\delta|X_{s-}|}\d s,
\end{split}
\end{equation}
where
\[K_2(\delta)=\max_{i\in\S}\int_{0<|z|<1}\!\!\Big(\e^{\delta |\sigma_i z|}-1+\delta|\sigma_iz|\Big)\nu(\d z)<\infty, \]
due to \eqref{a-1}. Furthermore,
\[\lim_{\delta\to 0}\frac{K_2(\delta)}{\delta}\leq \int_{0<|z|<1}\!\! 2\sigma_\ast |z|\nu(\d z)<\infty.\]

Inserting \eqref{f-4}, \eqref{f-5}, \eqref{f-6} into \eqref{f-3}, we obtain
\begin{equation*}
\frac{\d}{\d t}\E\e^{\delta|X_t|}\leq  M\delta+K_3(\delta) \E\e^{\delta|X_{t}|},
\end{equation*}
with $K_3(\delta)=K_1(\delta)+K_2(\delta)-c\delta$. By the arbitrariness of $c$, we can take $\delta>0$ sufficiently small  and $c>0$ large enough so that $K_3(\delta)<0$.
Applying Gronwall's inequality, we obtain
\begin{equation}\label{f-7}
\E\e^{\delta|X_t|}\leq \E\e^{\delta|X_0|}\e^{K_3(\delta) t}+M\delta\int_0^t\e^{K_3(\delta)(t-s)}\d s.
\end{equation}
The negativeness of $K_3(\delta)$ implies that
\begin{equation}\label{f-8}
\sup_{t\geq 0} \E\e^{\delta|X_t|}<\infty
\end{equation}
if $\E\e^{\delta|X_0|}<\infty$. In particular, we take $(X_0,\La_0)$ starting from some fixed point $(x_0,i)$, and the positive recurrence of $(X_t,\La_t)$ yields that
the distribution of $(X_t,\La_t)$ weakly converges to its stationary distribution $\pi$. Together with the estimate \eqref{f-8}, we finally get
\[\sum_{j\in \S}\int_{\R}\e^{\delta |x|}\pi(\d x, j)\leq \liminf_{t\to \infty} \E\e^{\delta |X_t|}<\sup_{t>0}\E\e^{\delta |X_t|}<\infty,
\] which is the desired conclusion.

Second, we proceed to show $\pi$ is heavy-tailed under conditions \eqref{a-4} and $\max_{i\in\S}\alpha_i>0$. For $p>2\vee \kappa$, thanks to It\^o's formula,
\begin{equation}\label{g-1}
\begin{split}
  &|X_t|^p-|X_0|^p\\
  &=\int_0^t\!\Big(p|X_{s-}|^{p-1}\sgn(X_{s-})
  (\alpha_{\La_{s-}}X_{s-}+b\sigma_{\La_{s-}})+\frac{ap(p-1)}{2}
  \sigma_{\La_{s-}}^2|X_{s-}|^{p-2}\Big)\d s\\
  &\quad+\!\int_0^t\!\int_{z\neq 0}\!\!\big(|X_{s-}\!+\!\sigma_{\La_{s-}}\!z|^p\!-\!|X_{s-}|^p\!-\!
  p\sigma_{\La_{s-}}z\,\sgn(X_{s-})|X_{s-}|^{p-1}\mathbf{1}_{0<|z|<1}\!\big)\nu(\d z)\d s\\
  &\quad+\!\int_0^t\!\!p\sqrt{a}\sigma_{\La_{s-}}\!\sgn(X_{s-})|X_{s-}|^{p-1}\d B_s\!+\!\int_0^t\!\!\int_{z\neq 0}\!\!\big(|X_{s-}\!\!+\!\sigma_{\La_{s-}}\!z|^p\!-\!|X_{s-}|^p\big)\wt N(\d s,\d z).
\end{split}
\end{equation}
Given the initial distribution $\gamma$ of $(X_0,\La_0)$, let $\F^\La=\sigma\{\La_s; 0\leq s<\infty\}$ and
\begin{equation}\label{g-2}
G_p^\gamma(t)=\E\big[|X_{t-}|^p\big|\F^\La\big],\quad t\geq 0,
\end{equation}
where the superscript $\gamma$ in $G_p^\gamma(t)$ emphasizes the initial distribution $\gamma$.
By virtue of the independence of $(\La_t)$ and $(Z_t)$, taking conditional expectation with respect to $\F^\La$ yields that
\begin{equation}\label{g-3}
\frac{\d G_p^\gamma(t)}{\d t}=p\alpha_{\La_{t-}}G_p^\gamma(t)+I(t),
\end{equation}
where $I(t)=I_1(t)+I_2(t)$ with
\begin{align*}
  I_1(t)&:=pb\sigma_{\La_{t-}}\E\big[\sgn(X_{t-})|X_{t-}|^{p-1}\big|\F^\La\big]
  +\frac{ap(p-1)\sigma_{\La_{t-}}^2}{2}G_{p-2}^\gamma(t),\\
  I_2(t)&:=\E\Big[\!\int_{z\neq 0}\!\!\big(|X_{t-}\!+\!\sigma_{\La_{t-}}\!z|^p\!-\!|X_{t-}|^p\!-\!
  p\sigma_{\La_{t-}}z\,\sgn(X_{t-})|X_{t-}|^{p-1}\mathbf{1}_{0<|z|<1}\!\big)\nu(\d z)\Big|\F^\La\Big].
\end{align*}
For $\theta>0$, $\veps>0$, and $0<r<p$, by H\"older's inequality and the following elementary inequality
\[ \theta x^{\frac rp}\leq \veps x+c_0,\quad x>0,\quad \text{with $c_0=\big(\veps/\theta\big)^{\frac{-r}{p+r}}$},\]
we have
\begin{equation}\label{g-4}
\theta G_r^{\gamma}(t)\leq \theta G_p^\gamma(t)^{\frac{r}p}\leq \veps G_p^\gamma(t)+c_0.
\end{equation}
Therefore, for any $\veps>0$ there exists a constant $C_1>0$ independent of $t$ such that
\begin{equation}\label{g-5}
|I_1(t)|\leq p\sigma_\ast |b|G_{p-1}^\gamma(t)+\frac{ap(p-1)\sigma_\ast^2}{2}G_{p-2}^\gamma(t)\leq \frac{\veps}{2}G_p^\gamma(t)+C_1.
\end{equation}

To estimate the term $I_2(t)$, we note that the mean value theorem implies
\[ |x+z|^p-|x|^p=p(\vartheta |x+z|+(1-\vartheta)|z|)^{p-1}(|x+z|-|x|), \quad x,z\in\R\]
for some $\vartheta\in [0,1]$ depending on $x,\,z$, and further
\begin{equation}\label{g-6}
\begin{split}
\big||x+z|^p-|x|^p\big|&\leq p\max\big\{|x+z|^{p-1}, |z|^{p-1}\big\}|z|\\
&\leq p\max\big\{2^{p-2}(|x|^{p-1}+|z|^{p-1}), |z|^{p-1}\big\}|z|.
\end{split}
\end{equation}
Besides, conditions \eqref{a-1} and \eqref{a-4} imply that for every $n\geq 1$,
$\int_{z\neq 0} |z|^n\nu(\d z)<\infty$. Finally, invoking \eqref{g-4} and \eqref{g-6}, we obtain that there exists some $C_2>0$ independent of $t$ such that
\begin{equation}\label{g-7}
|I_2(t)|\leq \frac{\veps}{2}G_p^\gamma(t)+C_2.
\end{equation}
Inserting the estimates \eqref{g-5}, \eqref{g-7} into \eqref{g-2}, we get
\begin{equation}\label{g-8}
\frac{\d G_p^\gamma(t)}{\d t}\geq (p\alpha_{\La_{t-}}-\veps)G_p^\gamma-M,
\end{equation}
where $M=C_1+C_2>0$.
Then
\begin{equation}\label{g-9}
\begin{split}
  G_p^\gamma(t)&\geq G_p^\gamma(0)\e^{\int_0^t(p\alpha_{\La_{s-}}+\veps)\d s}-M\int_0^t\e^{\int_s^t(p\alpha_{\La_{r-}}+\veps)\d r}\d s\\
  &\geq G_p^\gamma(0)\e^{\int_0^t(p\alpha_{\La_{s-}}+\veps)\d s}-M\int_0^t\e^{(p\max_{i\in\S}\alpha_i +\veps)(t-s)}\d s\\
  &=G_p^\gamma(0)\e^{\int_0^t(p\alpha_{\La_{s-}}+\veps)\d s}+\frac{M}{p\max_{i\in\S}\alpha_i+\veps}\big(\e^{(p\max_{i\in\S}\alpha_i+\veps)t}-1\big).
\end{split}
\end{equation}

We shall prove the statement (ii) by contradiction. Assume that $\sum_{i\in\S}\int_{\R}|x|^p\pi(\d x,i)<\infty$. Now let us take $\gamma$ in \eqref{g-2} to be the stationary distribution $\pi$ of $(X_t,\La_t)$, so for every $t>0$, the distribution of $(X_t,\La_t)$ is also $\pi$. Since $(\La_t)$ is an irreducible Markov chain in a finite state space, denote by $\mu$ its stationary distribution on $\S$. Then for every $i\in \S$, $\mu_i>0$. Moreover, it holds $\mu_i=\pi(\R,i)$ which can be seen by the weak convergence of the semigroup of $(X_t,\La_t)$ to $\pi$ and the semigroup of $(\La_t)$ to $\mu$. Consequently, we have
\begin{equation}\label{g-10}
\min_{i\in\S}\int_{\R}|x|^p\pi(\d x,i)>0.
\end{equation}
Taking expectation on both sides of \eqref{g-9}, we obtain
\begin{equation}\label{g-11}
\begin{split}
&\sum_{i\in\S}\int_{\R}|x|^p\pi(\d x,i)=\E|X_t|^p\\
&\geq \Big(\min_{i\in\S}\int_{\R}|x|^p\pi(\d x,i) \Big) \E\e^{\int_0^t(p\alpha_{\La_{s-}}+\veps)\d s}+\frac{M}{p\max_{i\in\S}\alpha_i+\veps}\big(\e^{(p\max_{i\in\S}\alpha_i+\veps)t}-1\big)\\
&\geq c_p\Big(\min_{i\in\S}\int_{\R}|x|^p\pi(\d x,i) \Big)\e^{(-\eta_p+ \veps) t}+\frac{M}{p\max_{i\in\S}\alpha_i+\veps}\big(\e^{(p\max_{i\in\S}\alpha_i+\veps)t}-1\big),
\end{split}
\end{equation}
where $c_p$ is a positive constant, $\eta_p<0$ is defined by \eqref{eta}, and in the last step we have used \cite[Proposition 4.1]{BGM10}.
Choosing $\veps>0$ such that $\eta_p+\veps<0$, and letting $t\to\infty$ in \eqref{g-11}, the left-hand side is finite, but the right-hand side goes to $\infty$, which is a contradiction. Therefore, the $p$-th moment of $\pi$ is infinite, and hence $\pi$ is heavy-tailed.
\fin


\begin{thebibliography}{1e}

\bibitem {A04} D. Applebaum, {L\'{e}vy processes and stochastic calculus.} Cambridge University Press. 2004.

\bibitem {BGM10} J. Bardet, H. Guerin, F. Malrieu, {Long time behavior of diffusions with Markov switching.} ALEA Lat. Am. J. Probab. Math. Stat., {\bf 7} (2010), 151-170.

\bibitem{BMR01} O.E. Barndorff-Nielsen, T. Mikosch, S.I. Resnick, {L\'{e}vy processes: Theory and Applications.} Birkh\"{a}user, Boston. 2001.

\bibitem{BS01} O.E. Barndorff-Nielsen, N. Shephard, {Non-Gaussian Ornstein-
Uhlenbeck-based models and some of their uses in financial economics.} Journal of the Royal Statistical Society, Series B, Statistical Methodology, {\bf 63} (2001), 167-207.

\bibitem{BD} S.J. Brown, P.H. Dybvig, {The empirical implications of the Cox, Ingersoll, Ross theory of the term structure of interest rates.} Journal of Finance {\bf 41} (1986), 617-630.

\bibitem{CPY} P. Carmona, F. Petit, M. Yor, {On exponential functionals of certain L\'{e}vy processes.} Stochastics and Stochastic Rep. {\bf 47} (1994), 71-101.

\bibitem{DY} B. De Saporta, J.F. Yao, {Tail of a linear diffusion with Markov switching.} Ann. Appl. Probab. {\bf 15}  (2005), 992-1018.

\bibitem {FKZ11} S. Foss, D. Korshunov, S. Zachary, {An Introduction to Heavy-Tailed and Subexponential Distributions.} Springer, New York. 2011.

\bibitem {KS91} I. Karatzas, S. Shreve, {Brownian Motion and Stochastic Calculus.} Springer, New York. 1991.

\bibitem {HS17} T. Hou, J. Shao, {Heavy tail and light tail of Cox-Ingersoll-Ross processes with regime-switching.} (2017), arXiv:1709.01691.

\bibitem{LM05} L. Lindnera, R. Maller, {L\'evy integrals and the stationarity of generalised Ornstein-Uhlenbeck processes.} Stoch. Proc. and their Appl. {\bf 115} (2005), 1701-1722.

\bibitem{Nov03} A. Novikov, {Martingales and first-exit times for the Ornstein-Uhlenbeck process with jumps.} Theory Probab. Appl. {\bf 48} (2003), 340-358.

\bibitem {PP93} M. Pinsky, R. Pinsky, {Transience/recurrence and central limit theorem behavior for diffusion in andom temporal environments.} Ann. Probab. {\bf 21} (1993), 433-452.

\bibitem {SY84} K.  Sato, M. Yamazato, {Operator-self-decomposable distributions as limit distributions of processes of Ornstein-Uhlenbeck type.} Stoch. Process. Appl. {\bf17} (1984), 73-100.

\bibitem{SWY} K. Sato, T. Watanabe, M. Yamazato, Recurrence conditions for multidimensional processes of Ornstein-Uhlenbeck type, J. Math. Soc. Japan 46 (1994), 245-265.

\bibitem{SWYY}K. Sato, T. Watanabe, K. Yamzmuro, M. Yamazato, Multidimentional process of Ornstein-Uhlenbeck type with nondiagonalizable matrix in linear drift terms, Nagoya Math. J. 141 (1996), 45-78.

\bibitem {Sh15} J. Shao, {Criteria for transience and recurrence of regime-switching diffusion processes.} Electron. J. Probab. {\bf 20, 63} (2015), 1-15.

\bibitem{Shi} T. Shiga, A recurrence criterion for Markov processes of Ornstein-Uhlenbeck type, Prob. Th. Rel. Fields, 85 (1990), 425-447.

\bibitem {Sko} A. V. Skorokhod, {Asymptotic methods in the theorem of stochastic differential equations.} American Mathematical Society, Providence, RI. 1989.

\bibitem{Wat} T. Watanabe, Sato's conjecture on recurrence conditions for multidimensional processes of Ornstein-Uhlenbeck type, J. Math. Soc. Japan, 50 (1998), 155-168.

\bibitem{YZ} G. Yin, C. Zhu, {Hybrid switching diffusions: properties and applications.} Vol. 63, Stochastic Modeling and Applied Probability, Springer, New York. 2010.

\end{thebibliography}
\end{document}